\documentclass[amssymb,12pt ]{article}
\pdfoutput=1
\usepackage{geometry}
\usepackage{amsfonts}
\usepackage{amssymb}
\usepackage{amsmath}
\usepackage{graphicx}
\usepackage{curves}
\usepackage{subfigure}
\newtheorem{thm}{Theorem}[section]
\newtheorem{cor}[thm]{Corollary}
\newtheorem{lem}[thm]{Lemma}
\newtheorem{pro}[thm]{Proposition}
\newenvironment{ack}{\noindent{\bf Acknowledgments}}

\newenvironment{proof}{\noindent{\bf Proof. }}{\hfill\framebox[2mm]{}}

\begin{document}

\title{Yokota theory, the invariant trace fields of hyperbolic knots and the Borel regulator map
}
\author{\sc Jinseok Cho}
\maketitle
\begin{abstract}
For a hyperbolic link complement with a triangulation, there are hyperbolicity equations of the triangulation, 
which guarantee the hyperbolic structure of the link complement.
In this paper, we explain that 
the number of the essential solutions of the equations is equal to or bigger than
the extension degree of the invariant trace field of the link.

On the other hand, Yokota suggested a potential function of a hyperbolic knot,
which gives the hyperbolicity equations and the complex volume of the knot.
Applying the fact above to his theory, we explain that the potential function also gives
all the values of the Borel regulator map and the complex volumes of the parabolic representations.
Furthermore, we explain the maximum value of the imaginary parts of the complex volumes
is the volume of the complete hyperbolic structure of the knot complement.
This enables us to calculate the complex volume of the knot complement combinatorially 
from the knot diagram in many cases. 

Especially, if the number of the essential solutions of the hyperbolicity equations
and the extension degree of the invariant trace field are the same, 
then the evaluation of all essential complex solutions of the hyperbolicity equations 
to the imaginary part of the potential function
is the same with the Borel regulator map. We show these actually happens in the case of the twist knots.
\end{abstract}

\section{Main Results}\label{sec1}

This article explains certain relations between the extended Bloch group theory and Yokota theory, 
and reports a new interpretation of Yokota theory using well-known results. 
Each parts of the results in this article are already known to different groups of researchers, 
but the author believes the whole story is not well-known to researchers.
Many contents rely on the lectures of Neumann in KIAS in spring 2010, which are also contained
in \cite{NeumannLecture}.

For a hyperbolic link $L$, we consider a topological ideal triangulation of $S^3-L$. To get a hyperbolic structure
of the link complement, we parameterize each tetrahedra with complex variables and then write down
certain equations which guarantee edge relations and cusp conditions as in Chapter 4 of \cite{Thurston}.
We call the set of equations {\it hyperbolicity equations}. 

If some variables in a solution of the hyperbolicity equations become 0, 1 or $\infty$, 
we call the solution {\it non-essential solution}. Note that if the hyperbolicity equations have an {\it essential solution},
each solution induces a parabolic representation of $\pi_1(S^3-L)$ into ${\rm PSL}(2,\mathbb{C})$. 
(See Section 6 of \cite{Takahashi85} for a reference.)
Therefore, if the hyperbolicity equations have an essential solution, then they have unique solution 
which gives the hyperbolic structure to the link complement by the Mostow rigidity.
We call the unique solution {\it geometric solution}.

Consider a triangulation of $S^3-L$ and let $\mathcal{H}$ be the set of hyperbolicity equations
with variables $x_1,\ldots, x_n$. 
Assume $\mathcal{H}$ has an essential solution and choose the geometric solution 
$(x_1^{(0)},\ldots,x_n^{(0)})$.
Note that, for a link complement, the invariant trace field and the trace field are the same. Furthermore,
in our consideration, they become $\mathbb{Q}(x_1^{(0)},\ldots,x_n^{(0)})$. (See \cite{Reid1} or \cite{Neumann00}).
We put the invariant trace field $k(L):=\mathbb{Q}(x_1^{(0)},\ldots,x_n^{(0)})$. Also note that
all $x_k^{(0)}$'s are algebraic numbers and the extension degree $[k(L):\mathbb{Q}]$ is a finite number.

Now we introduce a theorem. Note that this is not a new result, but was already known to several researchers.

\begin{thm}\label{thm}
Assume the hyperbolicity equations $\mathcal{H}$ have the geometric solution $(x_1^{(0)},\ldots,x_n^{(0)})$.
Then the number of the essential solutions of $\mathcal{H}$ is equal to or bigger than the extension degree $[k(L):\mathbb{Q}]$.
\end{thm}

\begin{proof} 
By the ``Theorem of the Primitive Elements", $k(L)=\mathbb{Q}[x]/\langle f(x)\rangle$ for some irreducible polynomial with degree $[k(L):\mathbb{Q}]$. Also the extension degree $[k(L):\mathbb{Q}]$
is the number of embeddings from $k(L)=\mathbb{Q}[x]/\langle f(x)\rangle$ to $\mathbb{C}$.
For convenience, we assume each $x_1^{(0)},\ldots,x_n^{(0)}$ are expressed by 
some elements of $\mathbb{Q}(x)$. (This assumption implies one of the embeddings is the identity map.)

Choose an embedding $\tau : k(L) \longrightarrow\mathbb{C}$. 
For any hyperbolicity equation $h(x_1,\ldots,x_n)=0\in\mathcal{H}$,
$$h(\tau(x_1^{(0)}),\ldots,\tau(x_n^{(0)}))=\tau(h(x_1^{(0)},\ldots,x_n^{(0)}))=0.$$
Therefore $(\tau(x_1^{(0)}),\ldots,\tau(x_n^{(0)}))$ becomes an essential solution of $\mathcal{H}$.

Furthermore, for another embedding $\tau^{\prime}$, 
if $(\tau(x_1^{(0)}),\ldots,\tau(x_n^{(0)}))=(\tau^{\prime}(x_1^{(0)}),\ldots,\tau^{\prime}(x_n^{(0)}))$,
then $\tau=\tau^{\prime}$ by the definition of $k(L)=\mathbb{Q}(x_1^{(0)},\ldots,x_n^{(0)})$.
Therefore, different embeddings of $k(L)$ gives different essential solutions of $\mathcal{H}$.

\end{proof}

A quick application of Theorem \ref{thm} is that the number of essential solutions of $\mathcal{H}$ becomes
an upper bound of the extension degree. For example, let $C(a_1,\ldots,a_m)$ be the 2-bridge link in Conway notation
satisfying $m\geq2, , ~a_1\geq 2,~a_m\geq 2$ and $a_k>0 ~(k=2,\ldots,m-1)$ as in \cite{Ohnuki05} and \cite{Sakuma95}. 
We remark the ideal triangulation of Sakuma-Weeks in \cite{Sakuma95} has the geometric solution because
it was shown in \cite{Sakuma95} that the triangulation is combinatorially equivalent to the canonical decomposition in \cite{Epstein88}.
Using the ideal triangulation, we obtain the following corollary.

\begin{cor}
The following inequality holds for the 2-bridge link $C(a_1,\ldots,a_m)$.
$$[k(C(a_1,\ldots,a_m)):\mathbb{Q}]\leq\left[\frac{\alpha(a_1,\ldots,a_m)-1}{2}\right],$$
where $\alpha(a_1,\ldots,a_m)$ is a positive integer defined by the recursive formula
$$\alpha(\emptyset)=1,~ \alpha(a_1)=a_1, ~
\alpha(a_1,\ldots,a_{k})=a_{k}\cdot\alpha(a_1,\ldots,a_{k-1})+\alpha(a_1,\ldots,a_{k-2}),$$
and $[x]$ is the integer part of $x\in \mathbb{R}$.
\end{cor}

\begin{proof}
This is a direct consequence of LEMMA II.5.8 of \cite{Sakuma95} and Theorem \ref{thm} above.

\end{proof}

In fact, the above application is not new because it was already shown in Chapter 4.5 of \cite{Reid1} that 
$$[k(C(a_1,\ldots,a_m)):\mathbb{Q}]\leq\left[\frac{p-1}{2}\right],$$
where $p,q$ are relatively prime positive integers satisfying
$$\frac{q}{p}=\cfrac{1}{a_1+\cfrac{1}{a_2+\cdots+\cfrac{1}{a_m}}}.$$
Note that $\alpha(a_1,\ldots,a_m)=p$ is a well-known property of continued fractions.

On the other hand, if we apply the idea of Theorem \ref{thm} to Yokota theory,
we can have a new insight on Yokota theory. We will explain it in the following.

Kashaev conjectured the following relation in \cite{Kashaev97} :
  $$\text{vol}(L)=2\pi \lim_{N\rightarrow\infty}\frac{\log|\langle L\rangle_N|}{N},$$
  where $L$ is a hyperbolic link, vol($L$) is the hyperbolic volume of $S^3-L$,
  $\langle L\rangle_N$ is the $N$-th Kashaev invariant.
After that, the generalized conjecture was proposed in \cite{Murakami02} that 
  $$i(\text{vol}(L)+i\,\text{cs}(L))\equiv 2\pi i\lim_{N\rightarrow\infty}\frac{\log\langle L\rangle_N}{N}~~({\rm mod}~\pi^2),$$
  where cs($L$) is the Chern-Simons invariant of $S^3-L$ defined in \cite{Meyerhoff86} 
  with the normalization of modulo $\pi^2$.
  
The calculation of the actual limit of the Kashaev invariant is very hard, and only few cases are known.
(The known results until now can be found in \cite{Veen}.)
On the other hand, while proposing the conjecture, Kashaev used some formal approximation to predict
the actual limit. His formal approximation was formulated as {\it optimistic limit} by H. Murakami in \cite{Murakami00b}.
Although the optimistic limit is not yet proved to be the actual limit of the Kashaev invariant,
Yokota made a very useful way to calculate the optimistic limit using his potential function.

For a hyperbolic knot $K$ and its diagram $D$ with certain conditions, 
Yokota defined a potential function $V(z_1,\ldots,z_n)$
and a triangulation of the knot complement such that the set of the hyperbolicity equations becomes
\begin{equation}\label{hypeq}
\mathcal{H}=\left\{\exp\left(z_k\frac{\partial V}{\partial z_k}\right)=1~:~k=1,\ldots,n\right\}.
\end{equation}
(We will explain the definition of $V(z_1,\ldots,z_n)$ in detail in Section \ref{sec2}.)
From now on, we always assume $\mathcal{H}$ has the geometric solution $(z_1^{(0)},\ldots,z_n^{(0)})$.
Let
$$V_0(z_1,\ldots,z_n):=V(z_1,\ldots,z_n)-\sum_{k=1}^n z_k\frac{\partial V}{\partial z_k}\log z_k.$$
Then Yokota proved, in \cite{YokotaPre}, that the optimistic limit of 
$ 2\pi i\frac{\log\langle L\rangle_N}{N}$ becomes
$V_0(z_1^{(0)},\ldots,z_n^{(0)})$\footnote{
We remark that the Kashaev invariant of a knot $K$ defined in \cite{YokotaPre} 
is the one of the mirror image $\overline{K}$ defined in \cite{Murakami01a}. 
This article follows the definition of \cite{YokotaPre}.} and 
$${\rm Im}V_0(z_1^{(0)},\ldots,z_n^{(0)})={\rm vol}(K).$$ 
After that, he generalized it in \cite{Yokota10} to
\begin{equation}\label{cv}
V_0(z_1^{(0)},\ldots,z_n^{(0)})\equiv i ({\rm vol}(K)+i\,{\rm cs}(K))\text{   (mod } \pi^2),
\end{equation}
using the extended Bloch group theory of \cite{Neumann04} and \cite{Zickert09}.

Our second application of Theorem \ref{thm} came from a question that what happens if we evaluate
other solutions of $\mathcal{H}$ to $V_0(z_1,\ldots,z_n)$. It turns out the result is related to 
the Borel regulator map and the complex volumes of the parabolic representations.

In Yokota theory, a solution $(z_1,\ldots,z_n)$ of $\mathcal{H}$ determines 
complex parameters $(t_1,\ldots,t_s)$ of ideal tetrahedra of Yokota triangulation. 
(This will be explained in detail in Section \ref{sec2}.) 
For a chosen solution $(z_1,\ldots,z_n)$ of $\mathcal{H}$, if $t_m\notin\mathbb{R}$ for some $m=1,\ldots,s$,
we call the solution {\it complex solution}.
Likewise, if $t_m\neq 0,1,\infty$ for all $m=1,\ldots,s$, then we call the solution {\it essential solution}.
It is a well-known fact that
if there exists an essential solution of $\mathcal{H}$, then there exists the geometric solution of $\mathcal{H}$.
(For reference, see Section 2.8 of \cite{Tillmann05}.)

Consider Yokota triangulation of $S^3-K$ as in \cite{YokotaPre} or \cite{Yokota10}. 
For the invariant trace field $k(K)=\mathbb{Q}[x]/\langle f(x)\rangle$ of the knot $K$, 
let $[k(K):\mathbb{Q}]=2r_1+r_2$ and 
$\tau_1,\overline{\tau}_1,\ldots,\tau_{r_1},\overline{\tau}_{r_1}$ be the complex embeddings
$k(K)\rightarrow\mathbb{C}$
and $\tau_{r_1+1}\ldots,\tau_{r_2}$ be the real embeddings $k(K)\rightarrow\mathbb{R}$.
Also let $(t_1^{(0)},\ldots,t_s^{(0)})$ be the parameters of the ideal tetrahedra in Yokota triangulation, which give
the complete hyperbolic structure to $S^3-K$. 
Then, $k(K)=\mathbb{Q}(t_1^{(0)},\ldots,t_s^{(0)})$. 
We assume each $t_1^{(0)},\ldots,t_s^{(0)}$ are expressed by some elements of $\mathbb{Q}(x)$.
As explained in \cite{Yokota10} or in Section \ref{sec2}, $z_1^{(0)},\ldots,z_n^{(0)}$ can be expressed by
fractions of $t_1^{(0)},\ldots,t_s^{(0)}$, 
so we can also assume each $z_1^{(0)},\ldots,z_n^{(0)}$ are expressed by some elements of $\mathbb{Q}(x)$.
For $j=1,\ldots,r_1$, the $j$-th component of the Borel regulator map is defined by
$${\rm Borel}(S^3-K)_j:=\sum_{m=1}^{s}D_2(\tau_j(t_m^{(0)})),$$
where $D_2(t)={\rm Im~Li_2}(t)+\log|t|\arg(1-t)$ for $t\in \mathbb{C}-\{0,1\}$ is the Bloch-Wigner function 
and ${\rm Li}_2(t)=-\int_0^t\frac{\log(1-z)}{z}dz$ is the dilogarithm function.
(See \cite{Neumann00} or \cite{Neumann99} for details.)

\begin{cor}\label{cor}
Let $z_k^{(j)}:=\tau_j(z_k^{(0)})$ for $j=1,\ldots,r_1$. Then
$${\rm Im } V_0(z_1^{(j)},\ldots,z_n^{(j)})={\rm Borel}(S^3-K)_j.$$
Especially, if 
\begin{equation}\label{number}
[k(K):\mathbb{Q}]=(\text{the number of the essential solutions of }\mathcal{H}),
\end{equation}
then ${\rm Im} V_0(z_1,\ldots,z_n)$ becomes a plus or minus of an element of the Borel regulator map
for any essential complex solution $(z_1,\ldots,z_n)$ of $\mathcal{H}$.
\end{cor}

By Theorem \ref{thm}, we know that all the embeddings $k(K)\rightarrow\mathbb{C}$
are induced by some essential solutions of $\mathcal{H}$, but
there can exist some extra essential solutions which do not induce the embeddings.
However, the condition (\ref{number}) guarantees that there are no extra essential solutions.
We remark that there are many examples satisfying this condition.
Especially we will show that the twist knots actually satisfy (\ref{number}) in Section \ref{sec4}.

Although we will introduce a simple proof of Corollary \ref{cor} in Section \ref{sec3},
we remark Corollary \ref{cor} was already suggested in \cite{Zickert09} in a more general form as follows.

\begin{cor}\label{cor2}
Let $\mathrm{z}:=(z_1,\ldots,z_n)$ be any essential solution of $\mathcal{H}$ and
$\rho_{\mathrm{z}}:\pi_1(S^3-K)\rightarrow{\rm PSL}(2,\mathbb{C})$
be the parabolic representation induced by $\mathrm{z}$. Then
$$V_0(\mathrm{z})\equiv i({\rm vol}(\rho_{\mathrm{z}})+i\,{\rm cs}(\rho_{\mathrm{z}}))\text{ \rm  (mod } \pi^2),$$
where ${\rm vol}(\rho_{\mathrm{z}})+i\,{\rm cs}(\rho_{\mathrm{z}})$ is the complex volume of $\rho_{\mathrm{z}}$
defined in \cite{Zickert09}. Furthermore, for any essential solution $\mathrm{z}$ and the geometric solution
$\mathrm{z}^{(0)}:=(z_1^{(0)},\ldots,z_n^{(0)})$, the following inequality holds:
\begin{equation}\label{coreq}
  |{\rm Im} V_0(\mathrm{z})|\leq{\rm Im} V_0(\mathrm{z}^{(0)})={\rm vol}(S^3-K).
\end{equation}
\end{cor}

\begin{proof}
Yokota, in \cite{Yokota10}, used Zickert's formula of \cite{Zickert09} to prove the identity (\ref{cv}),
but Zickert's formula also holds for any parabolic representation $\rho_{\mathrm{z}}$.
Therefore, the first statement was already proved in \cite{Yokota10}.

On the other hand, it was proved in Corollary 5.10 of \cite{Francaviglia04} that, 
for any representation $\rho:\pi_1(S^3-K)\rightarrow{\rm PSL}(2,\mathbb{C})$,
$$|{\rm vol}(\rho)|\leq{\rm vol}(S^3-K),$$
so the inequality (\ref{coreq}) holds trivially.

\end{proof}

We remark that (\ref{coreq}) gives us a tool to determine
the volume of a hyperbolic knot complement combinatorially from the knot diagram using Yokota theory.
Furthermore, from Gromov-Thurston-Goldman rigidity in \cite{Dunfield99},
we can also obtain the Chern-Simons invariant combinatorially.

Although Yokota theory needs several assumptions, this method can be useful for many cases.
For example, if we apply it to the twist knots, we obtain the same potential functions $V(z_0,\ldots,z_n)$ 
as in (2.3) of \cite{Cho09b}.
We already know the existence of an essential solution of the twist knots. (This is explained in Section \ref{sec4}.) 
By the virtue of Gromov-Thurston-Goldman rigidity,
we can find the geometric solution by evaluating all the essential solutions
of (2.1) in \cite{Cho09b} to $V(z_0,\ldots,z_n)$ and picking up the unique solution that gives the maximal volume.
We can obtain the complex volume of the knot complement by evaluating the geometric solution to the potential function $V(z_0,\ldots,z_n)$.

\section{Definition of $V(z_1,\ldots,z_n)$ in Yokota theory}\label{sec2}

In this section, we explain the way to define $V(z_1,\ldots,z_n)$ following \cite{Yokota10}, \cite{YokotaPre} with an example Figure \ref{pic1}.
Note that Figure \ref{pic1} was already appeared in \cite{YokotaPre} as Figure 9.

\begin{figure}
\centering
  \subfigure[Knot]
  {\includegraphics[scale=0.4]{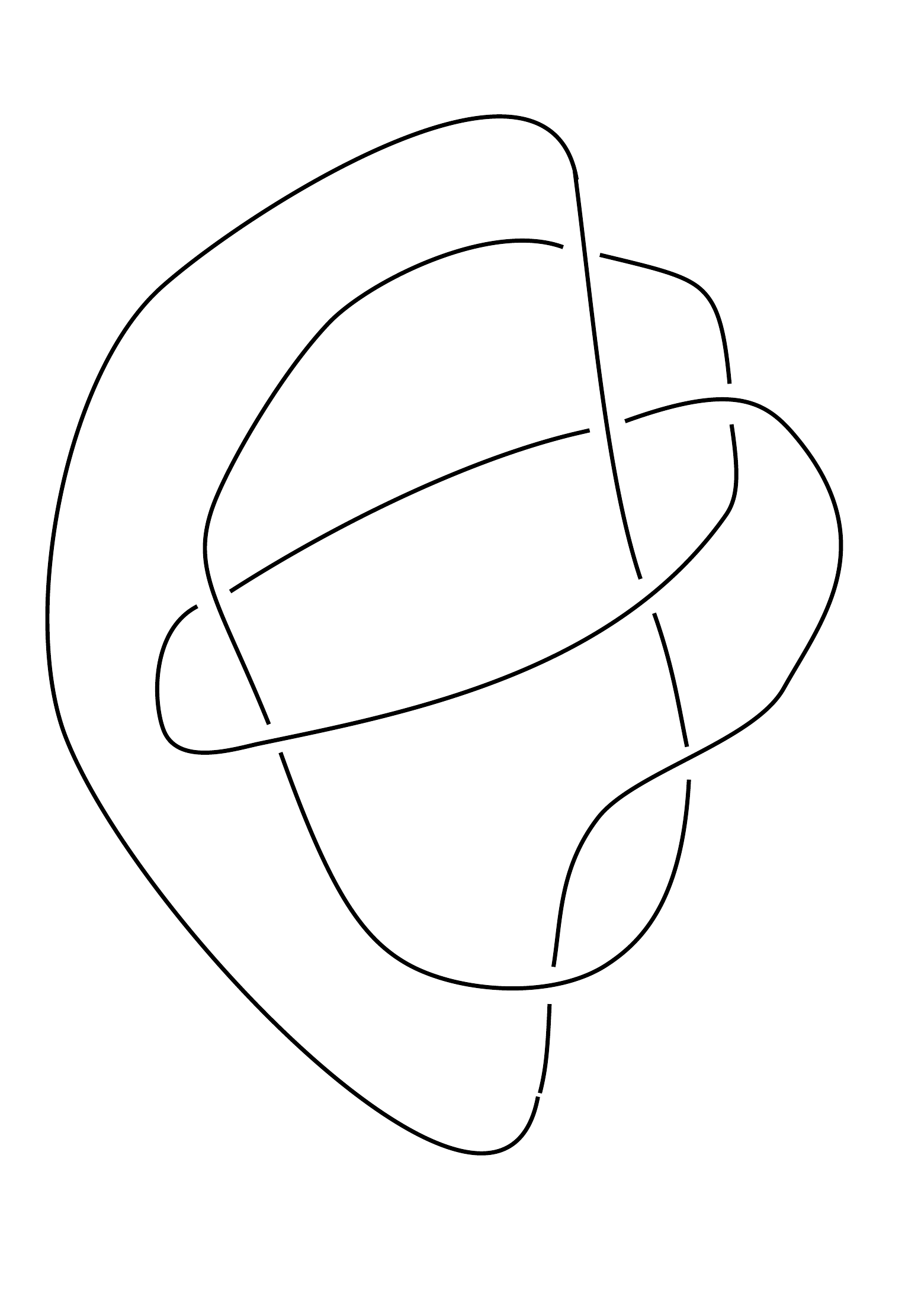}}
  \subfigure[(1,1)-tangle]
  {\includegraphics[scale=0.5]{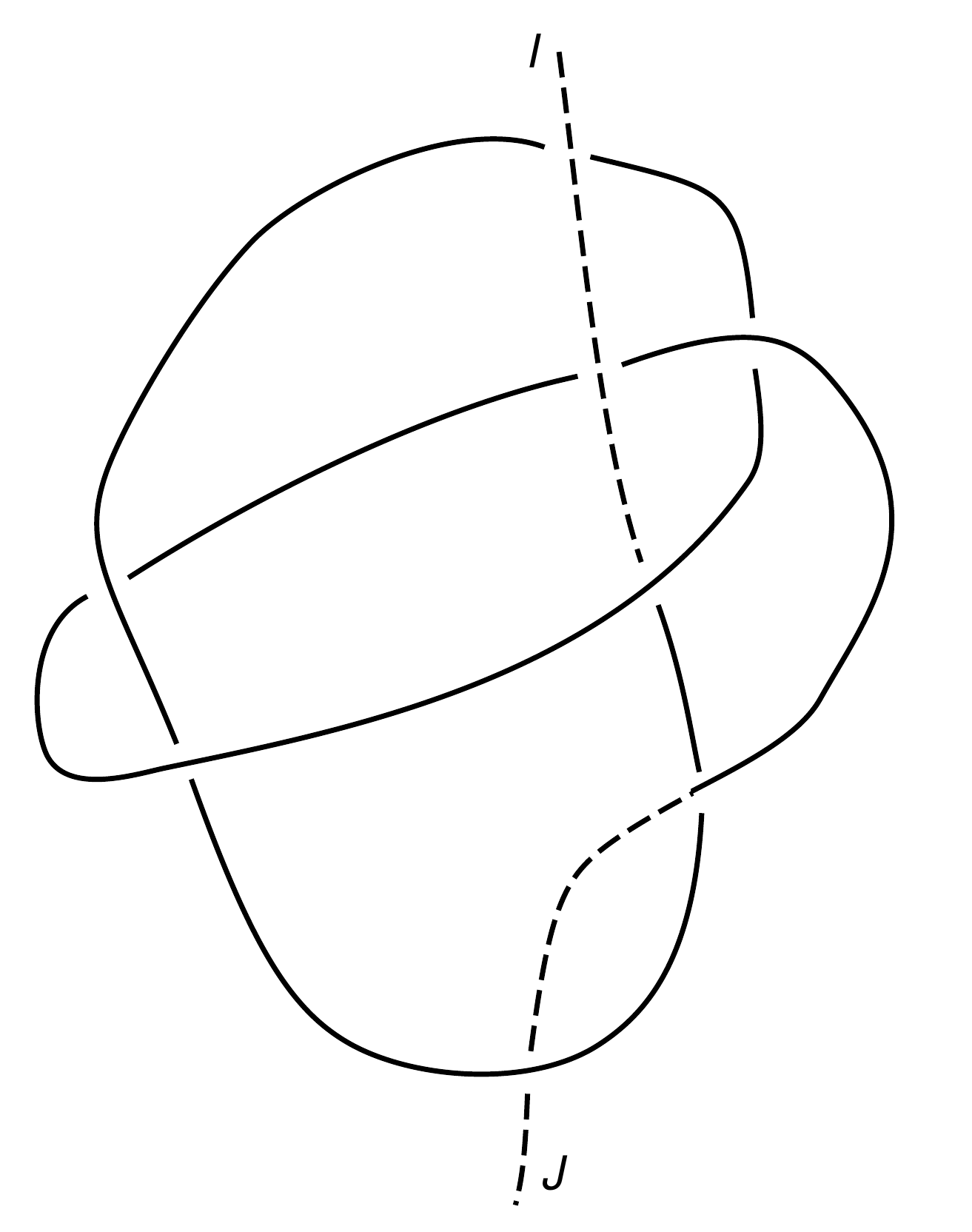}}\\
  \subfigure[$G$ is obtained by removing $I\cup J$]
  {\includegraphics[scale=0.5]{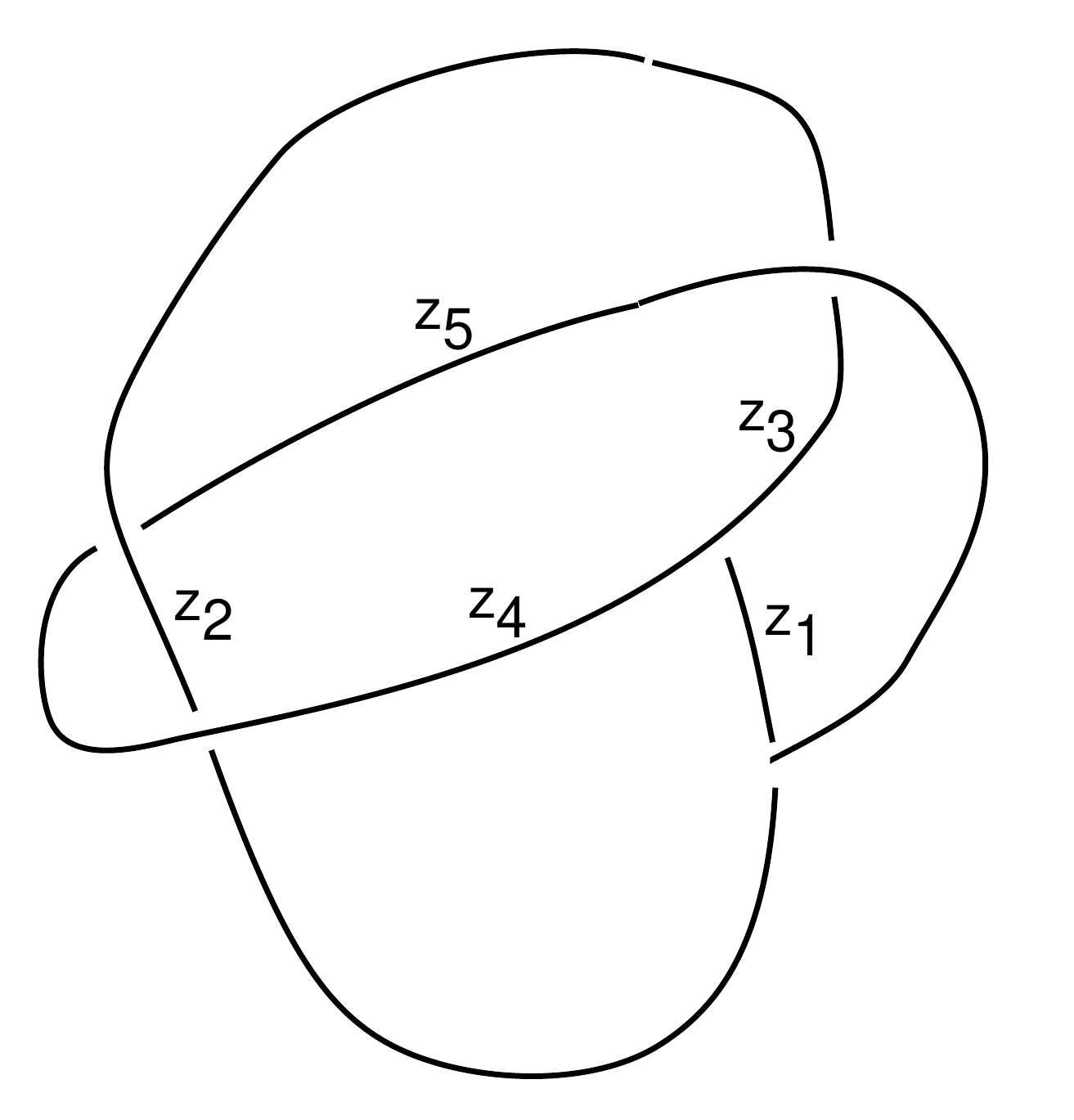}}
  \subfigure[$G$ with tetrahedra]
  {\includegraphics[scale=0.5]{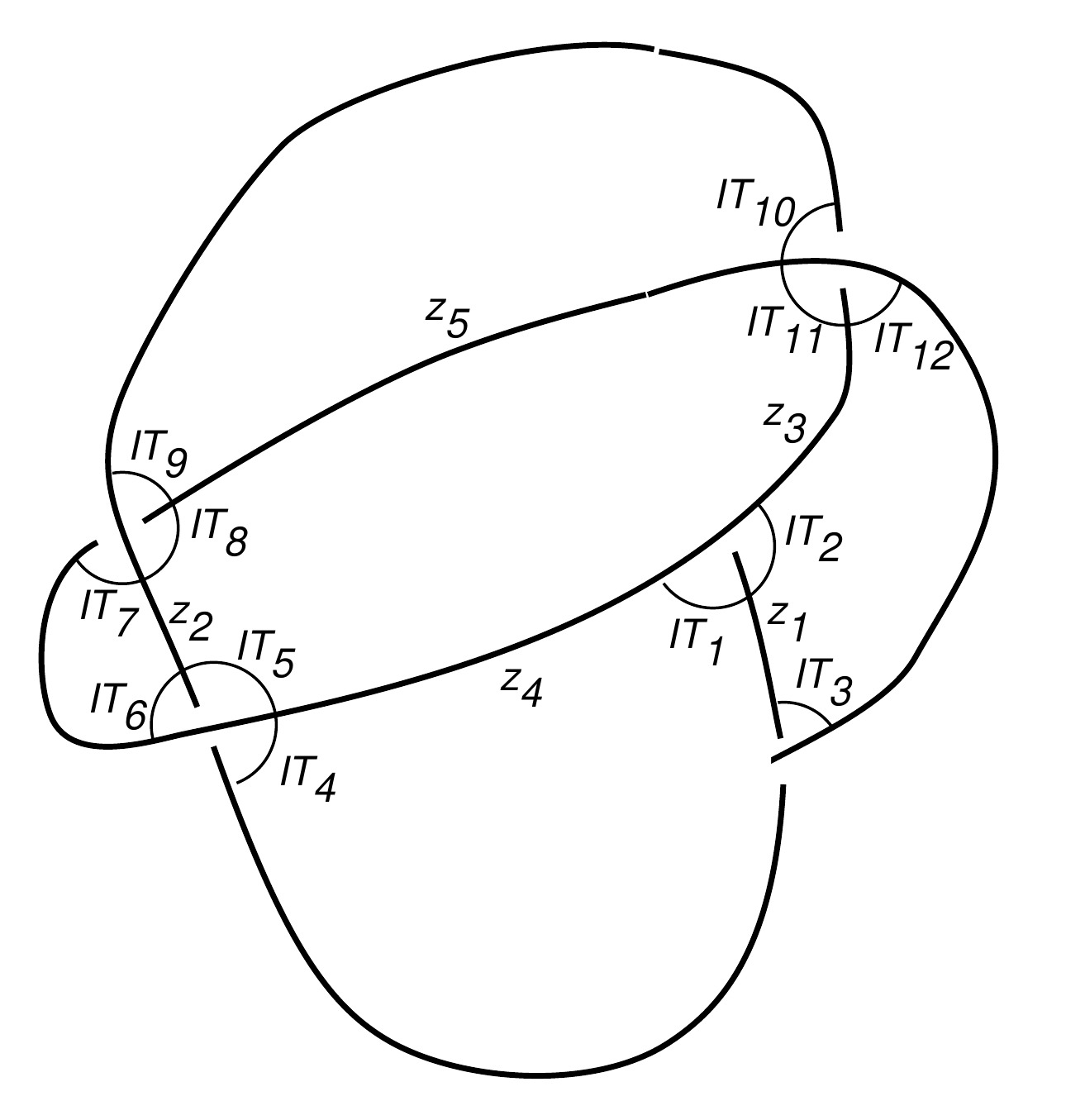}}
  \caption{Example}\label{pic1}
\end{figure}

Consider a hyperbolic knot $K$ and its diagram $D$. (See Figure \ref{pic1}(a).)
We define {\it sides} of $D$ as arcs connecting two adjacent crossing points.
For example, Figure \ref{pic1}(a) has 16 sides.

Now split a side of $D$ open so as to make a (1,1)-tangle diagram. (See Figure \ref{pic1}(b).)
Yokota assumed several conditions on the (1,1)-tangle diagrams.
(For detail, see Assumptions in \cite{Yokota10}.)
The Assumptions roughly means that we perform first and second Reidemeister moves on the tangle diagram
to reduce the crossing numbers as much as possible.
Also, let the two open sides be $I$ and $J$. Assume $I$ and $J$ are in an over-bridge and an under-bridge
respectively. Now extend $I$ and $J$ so that non-boundary endpoints of $I$ and $J$ become the first under-crossing point
and the last over-crossing point respectively, as in Figure \ref{pic1}(b). 
Then we assume the two non-boundary endpoints
of $I$ and $J$ are not the same. 
Yokota proved that we can always choose $I$ and $J$ with these conditions because, 
if not, then the knot should be the trefoil knot, which is not hyperbolic.
(See \cite{Yokota10} for details.)

We remove $I$ and $J$ on the tangle diagram and let the result be $G$. (See Figure \ref{pic1}(c).)
Note that, by removing $I\cup J$, some sides are glued together. 
(We consider the two trivalent points do not glue any sides.)
For example, in Figure \ref{pic1}(c), $G$ has 9 sides.
We define {\it contributing sides} as sides of $G$ which are not on the unbounded regions.
For example, Figure \ref{pic1}(c) has 5 contributing sides and 4 non-contributing sides.
We assign complex variables $z_1,\ldots,z_n$ to contributing sides and the real number 1 to non-contributing sides.

Now we draw small circles on each crossings and the trivalent points of $G$. 
Then remove some arcs of the circle that is
in the unbounded regions. Also remove two arcs that was on $I\cup J$. 
(See Figure \ref{pic1}(d) for the result.)
In this diagram, the survived arcs represent ideal tetrahedra and we can obtain an ideal triangulation of $S^3-K$
by gluing these tetrahedra. (See \cite{Yokota10} or \cite{YokotaPre} for gluing rules and details.)
We label each ideal tetrahedra $IT_1,IT_2,\ldots,IT_s$ and assign $t_m$ ($m=1,\ldots,s$)
as the complex parameter of $IT_m$.
We define $t_m$ as the counterclockwise ratio of the two adjacent sides of $IT_m$.
For example, in Figure \ref{pic1}(d),
\begin{eqnarray*}
t_1=\frac{z_1}{z_4},~ t_2=\frac{z_3}{z_1},~ t_3=\frac{z_1}{1},~ 
t_4=\frac{z_4}{1},~ t_5=\frac{z_2}{z_4}, ~t_6=\frac{1}{z_2},\\
t_7=\frac{z_2}{1},~ t_8=\frac{z_5}{z_2},~ t_9=\frac{1}{z_5},~
t_{10}=\frac{z_5}{1},~ t_{11}=\frac{z_3}{z_5}, ~t_{12}=\frac{1}{z_3}.
\end{eqnarray*}

Now we define the potential function $V(z_1,\ldots,z_n)$. For each tetrahedron $IT_m$, we assign dilogarithm
function as in Figure \ref{pic2}. Then $V(z_1,\ldots,z_n)$ is defined by 
the summation of all these dilogarithm functions.
We also define the sign $\sigma_m$ of $T_m$ by 
$$\sigma_m=\left\{\begin{array}{ll}
  1&\text{ if }~IT_m\text{ lies as in Figure \ref{pic2}(a)},\\
  -1&\text{ if }~IT_m\text{ lies as in Figure \ref{pic2}(b)}.
  \end{array}\right.$$
Then $V(z_1,\ldots,z_n)$ is expressed by
\begin{equation}\label{eq1}
V(z_1,\ldots,z_n)=\sum_{m=1}^s \sigma_m\left({\rm Li}_2(t_m^{\sigma_m})-\frac{\pi^2}{6}\right).
\end{equation}
  
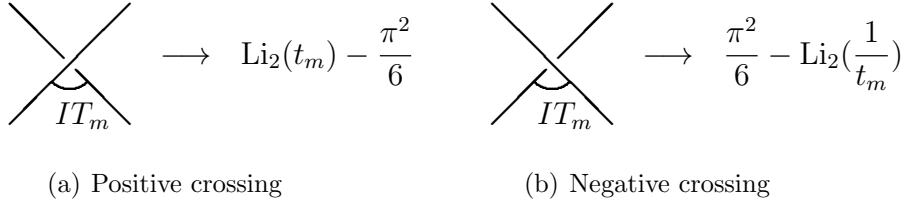
\begin{figure}
\centering
  \subfigure[Positive crossing]
  { {\setlength{\unitlength}{0.4cm}
  \begin{picture}(15,6)\thicklines
    \put(6,5){\line(-1,-1){4}}
    \put(6,1){\line(-1,1){1.8}}
    \put(3.8,3.2){\line(-1,1){1.8}}
    \put(4,3){\arc(-0.6,-0.6){90}}
    \put(3.5,1){$IT_m$}
    \put(7,3){$\displaystyle\longrightarrow ~~{\rm Li}_2(t_m)-\frac{\pi^2}{6}$}
  \end{picture}}}
  \subfigure[Negative crossing]
  { {\setlength{\unitlength}{0.4cm}
  \begin{picture}(15,6)\thicklines
    \put(2,5){\line(1,-1){4}}
    \put(2,1){\line(1,1){1.8}}
    \put(4.2,3.2){\line(1,1){1.8}}
    \put(4,3){\arc(-0.6,-0.6){90}}
    \put(3.5,1){$IT_m$}
    \put(7,3){$\displaystyle\longrightarrow ~~\frac{\pi^2}{6}-{\rm Li}_2(\frac{1}{t_m})$}
  \end{picture}}}
  \caption{Assignning dilogarithm functions to each tetrahedra}\label{pic2}
\end{figure}

For example, in Figure \ref{pic1}(d),
$$\sigma_1=\sigma_3=\sigma_5=\sigma_8=\sigma_{11}=1,
~\sigma_2=\sigma_4=\sigma_6=\sigma_7=\sigma_9=\sigma_{10}=\sigma_{12}=-1,$$
and
\begin{eqnarray*}
V(z_1,\ldots,z_5)=
{\rm Li}_2(\frac{z_1}{z_4})- {\rm Li}_2(\frac{z_1}{z_3})+ {\rm Li}_2(z_1)
-{\rm Li}_2(\frac{1}{z_4})+ {\rm Li}_2(\frac{z_2}{z_4})-{\rm Li}_2(z_2)\\
-{\rm Li}_2(\frac{1}{z_2})+{\rm Li}_2(\frac{z_5}{z_2})- {\rm Li}_2(z_5)
-{\rm Li}_2(\frac{1}{z_5})+{\rm Li}_2(\frac{z_3}{z_5})-{\rm Li}_2(z_3)+\frac{\pi^2}{3}.
\end{eqnarray*}

Note that $t_m$'s are ratios of $z_k$'s. Yokota, in Section 2.3 of \cite{Yokota10}, explained there is one-to-one correspondence between the complex parameters $\{t_m ~|~m=1,\ldots,s\}$ with certain conditions 
and a solution $\{z_k ~|~ k=1,\ldots,n\}$ of $\mathcal{H}$ in (\ref{hypeq}).
Therefore, if $(t_1^{(0)},\ldots,t_s^{(0)})$ of the ideal triangulation gives the hyperbolic structure to
$S^3-K$, we can find unique $(z_1^{(0)},\ldots,z_n^{(0)})$ which corresponds to $(t_1^{(0)},\ldots,t_s^{(0)})$.
Furthermore, the invariant trace field $k(K)=\mathbb{Q}(t_1^{(0)},\ldots,t_s^{(0)})$ coincides with 
$\mathbb{Q}(z_1^{(0)},\ldots,z_n^{(0)})$.
We call both of $(z_1^{(0)},\ldots,z_n^{(0)})$ and $(t_1^{(0)},\ldots,t_s^{(0)})$ {\it geometric solutions}.

\section{Proof of Corollary \ref{cor}}\label{sec3}

The following proposition and the proof were already appeared in \cite{YokotaPre} 
as Proposition 2.6, but we introduce them here for convenience.

\begin{pro}\label{prop}
Let $V$ be the potential function of a hyperbolic knot diagram $D$.
Let $(z_1,\ldots,z_n)$ be an essential solution of 
$\mathcal{H}=\left\{\exp\left(z_k\frac{\partial V}{\partial z_k}\right)=1~:~k=1,\ldots,n\right\}$
and $(t_1,\ldots,t_s)$ be the parameters of ideal tetrahedra corresponding to $(z_1,\ldots,z_n)$. 
Then
$${\rm Im} V_0(z_1,\ldots,z_n)=\sum_{m=1}^s D_2(t_m),$$
where $D_2(t)={\rm Im~Li_2}(t)+\log|t|\arg(1-t)$ for $t\in \mathbb{C}-\{0,1\}$ is the Bloch-Wigner function.
\end{pro}

\begin{proof}
By the property $D_2(t)=-D_2(\frac{1}{t})$, it is enough to show
\begin{equation*}
{\rm Im} V_0(z_1,\ldots,z_n)-\sum_{m=1}^s \sigma_m D_2(t_m^{\sigma_m})=0.
\end{equation*}
Since $(z_1,\ldots,z_n)$ is an essential solution of $\mathcal{H}$, we know
$${\rm Re}\left(z_k \frac{\partial V}{\partial z_k}\right)=0,$$
for $k=1,\ldots,n$, and therefore
$$\sum_{k=1}^n{\rm Re}\left(z_k \frac{\partial V}{\partial z_k}\right){\rm Im}(\log z_k)=0.$$
Using the above and (\ref{eq1}), we have
$${\rm Im} V_0(z_1,\ldots,z_n)
=\sum_{m=1}^s {\rm Im}\left(\sigma_m{\rm Li}_2(t_m^{\sigma_m})\right)
+\sum_{k=1}^n{\rm Im}\left(z_k\frac{\partial V}{\partial z_k}\right)\log |z_k|.$$
Therefore, by the definition of $D_2(t)$,
\begin{eqnarray}
&&{\rm Im} V_0(z_1,\ldots,z_n)-\sum_{m=1}^s \sigma_m D_2(t_m^{\sigma_m})\nonumber\\
&&=\sum_{k=1}^n{\rm Im}\left(z_k\frac{\partial V}{\partial z_k}\right)\log |z_k|
-\sum_{m=1}^s \sigma_m\log|t_m|\arg(1-t_m^{\sigma_m}).\label{eq3}
\end{eqnarray}

Note that there are four possible cases of the position of $IT_m$ as in Figure \ref{pic3}.
(We allow $z_l=1$ in Figure \ref{pic3}.)

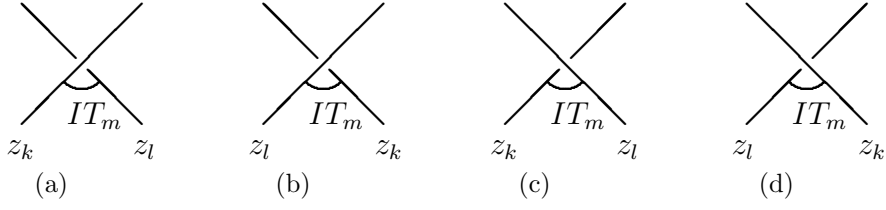
\begin{figure}[h]
\centering
  \subfigure[]
  { {\setlength{\unitlength}{0.4cm}
  \begin{picture}(7,6)\thicklines
    \put(6,5){\line(-1,-1){4}}
    \put(6,1){\line(-1,1){1.8}}
    \put(3.8,3.2){\line(-1,1){1.8}}
    \put(4,3){\arc(-0.6,-0.6){90}}
    \put(3.5,1){$IT_m$}
    \put(1.5,0){$z_k$}
    \put(5.7,0){$z_l$}
  \end{picture}}}
  \subfigure[]
  { {\setlength{\unitlength}{0.4cm}
  \begin{picture}(7,6)\thicklines
    \put(6,5){\line(-1,-1){4}}
    \put(6,1){\line(-1,1){1.8}}
    \put(3.8,3.2){\line(-1,1){1.8}}
    \put(4,3){\arc(-0.6,-0.6){90}}
    \put(3.5,1){$IT_m$}
    \put(1.5,0){$z_l$}
    \put(5.7,0){$z_k$}
  \end{picture}}}
  \subfigure[]
  { {\setlength{\unitlength}{0.4cm}
  \begin{picture}(7,6)\thicklines
    \put(2,5){\line(1,-1){4}}
    \put(2,1){\line(1,1){1.8}}
    \put(4.2,3.2){\line(1,1){1.8}}
    \put(4,3){\arc(-0.6,-0.6){90}}
    \put(3.5,1){$IT_m$}
    \put(1.5,0){$z_k$}
    \put(5.7,0){$z_l$}
  \end{picture}}}
  \subfigure[]
  { {\setlength{\unitlength}{0.4cm}
  \begin{picture}(7,6)\thicklines
    \put(2,5){\line(1,-1){4}}
    \put(2,1){\line(1,1){1.8}}
    \put(4.2,3.2){\line(1,1){1.8}}
    \put(4,3){\arc(-0.6,-0.6){90}}
    \put(3.5,1){$IT_m$}
    \put(1.5,0){$z_l$}
    \put(5.7,0){$z_k$}
  \end{picture}}}
  \caption{Four possible cases of the position of $IT_m$}\label{pic3}
\end{figure}

In the case of Figure \ref{pic3}(a), $\left(z_k\frac{\partial V}{\partial z_k}\right)$ contains a term $(\log(1-t_m))$, 
which corresponds to the differential of the dilogarithm function associated to $IT_m$ in Figure \ref{pic2}(a). 
Also, by definition, $t_m=\frac{z_l}{z_k}$ and $\sigma_m=1$,
so $\left(\sigma_m\log|t_m|\arg(1-t_m^{\sigma_m})\right)$ contains a term $(\log|z_k|\arg(1-t_m))$.
Therefore, the coefficient of $(\log|z_k|)$ corresponding to $IT_m$ in (\ref{eq3}) should be zero.

The other three cases can be easily verified with the same method. 
These implies the coefficient of any $\log|z_k|$ is zero, so (\ref{eq3}) becomes zero.

\end{proof}

Now we move to our original goal, the proof of Corollary \ref{cor}. 
We know $(z_1^{(j)},\ldots,z_n^{(j)})$ is an essential solution of $\mathcal{H}$
by Theorem \ref{thm}.
Let $(t_1^{(j)},\ldots,t_s^{(j)})$ be the parameters of the ideal tetrahedra $IT_1,\ldots, IT_s$
corresponding to $(z_1^{(j)},\ldots,z_n^{(j)})$.
(Recall $t_m^{(j)}$ is the counterclockwise ratio of $z_k^{(j)}$.)
Then, using Proposition \ref{prop}, we obtain
$${\rm Im} V_0(z_1^{(j)},\ldots,z_n^{(j)})=\sum_{m=1}^{s}D_2(t_m^{(j)})
=\sum_{m=1}^{s}D_2(\tau_j(t_m^{(0)}))={\rm Borel}(S^3-K)_j,$$
and the proof is completed.

\section{Properties of twist knots}\label{sec4}

Let $K$ be the twist knot $C(2,n+1)$ in Conway notation for $(n\geq 1)$. 
According to Section 7 of \cite{Cho10a},
$C(2,n+1)$ has $n$ contributing sides $z_1,\ldots,z_n$, and the potential function becomes
$$V(z_1,\ldots,z_n)={\rm Li}_2(\frac{1}{z_1})+\sum_{k=2}^n\left\{\frac{\pi^2}{6}
-{\rm Li}_2(z_{k-1})+{\rm Li}_2(\frac{z_{k-1}}{z_k})-{\rm Li}_2(\frac{1}{z_{k}})\right\}-{\rm Li}_2(z_n).$$
Also, the elements of the hyperbolicity equations 
$\mathcal{H}=\left\{\exp\left(z_k\frac{\partial V}{\partial z_k}\right)=1\right\}$ becomes
\begin{eqnarray}
1-\frac{z_1}{z_2}&=&1-\frac{1}{z_1}-z_1+1,\label{eq4}\\
1-\frac{z_k}{z_{k+1}}+\frac{1}{z_{k+1}}-\frac{1}{z_{k}}&=&1-\frac{z_{k-1}}{z_k}-z_k+z_{k-1} ~\text{ for }k=2,3,\ldots,n-1,\label{eq5}\\
1-\frac{1}{z_{n}}&=&1-\frac{z_{n-1}}{z_{n}}-z_n+z_{n-1}.\label{eq6}
\end{eqnarray}

Note that the existence of the geometric solution of $\mathcal{H}$ was already known as follows:
the triangulation of Sakuma-Weeks in \cite{Sakuma95} has the geometric solution and 
the hyperbolicity equations of Ohnuki's triangulation in \cite{Ohnuki05}
coincide with Sakuma-Weeks' one. Therefore, Ohnuki's triangulation has an essential solution
and the geometric solution.
Furthermore, the equation (7.5) of \cite{Cho10a} showed the relation between Ohnuki's geometric solution
and Yokota's one.

In this section, we will show 
$$[k(C(2,n+1)):\mathbb{Q}]=(\text{the number of the essential solutions of }\mathcal{H}).$$
Note that $[k(C(2,n+1)):\mathbb{Q}]=n+1$ was already proved in \cite{Hoste01}, so we will focus on
the fact 
$$(\text{the number of the essential solutions of }\mathcal{H})\leq n+1.$$

\begin{lem}\label{lem2}
Let $(z_1,\ldots,z_n)$ be an essential solution of (\ref{eq4}), (\ref{eq5}), (\ref{eq6}). Then, for $k=3,4,\ldots,n-1$,
$$\frac{1}{z_k}=1-z_{k-2}+z_n.$$
\end{lem}

\begin{proof}
Consider the equation (\ref{eq5}) for $k,k+1,\ldots,n-1$ and add all of them up with (\ref{eq6}). Then we obtain
\begin{equation}\label{eq7}
1-\frac{1}{z_k}=1-\frac{z_{k-1}}{z_k}+z_{k-1}-z_n.
\end{equation}
On the other hand, (\ref{eq5}) can be expressed by
$$1-\frac{z_{k-1}}{z_{k}}=-z_{k-1}\left(1-\frac{z_{k-2}}{z_{k-1}}\right).$$
Applying it to (\ref{eq7}), we obtain the result.

\end{proof}

To complete the proof, we will show that all $z_k$'s can be expressed by $z_1$ and $z_1$ satisfies
a polynomial equation with degree at most $n+1$. At first, summation of all (\ref{eq4}), (\ref{eq5}), (\ref{eq6}) gives
\begin{equation*}
\frac{1}{z_1}-\frac{1}{z_2}=1-z_n.
\end{equation*}
Applying (\ref{eq4}) in the form $\frac{1}{z_1}-\frac{1}{z_2}=-\left(1-\frac{1}{z_1}\right)^2$ to the above,
we obtain
\begin{equation}\label{eq8}
z_n=1+\left(1-\frac{1}{z_1}\right)^2.
\end{equation}

Note that
\begin{eqnarray*}
\frac{1}{z_2}=\frac{1}{z_1}+\left(1-\frac{1}{z_1}\right)^2=\frac{z_1^2-z_1+1}{z_1^2},\\
\frac{1}{z_3}=1-z_1+z_n=\frac{-z_1^3+3z_1^2-2z_1+1}{z_1^2}.
\end{eqnarray*}
Using the above, (\ref{eq8}), Lemma \ref{lem2} and the induction on $k$, any $z_k$ can be expressed by $z_1$.
Furthermore, we can express $\frac{1}{z_k}$ by
$$\frac{1}{z_k}=\frac{p_k(z_1)}{q_k(z_1)},$$
for $k=2,\ldots,n-1$, where 
$p_k(z_1), q_k(z_1)\in\mathbb{Z}[z_1]$ are polynomials with degree at most $k$.
Applying $\frac{1}{z_{n-1}}=\frac{p_{n-1}(z_1)}{q_{n-1}(z_1)}$ 
and (\ref{eq8}) to (\ref{eq6}) in the form $z_n=z_{n-1}-1$, we obtain
$$z_1^2 q_{n-1}(z_1)=(3z_1^2-2z_1+1)p_{n-1}(z_1).$$
Therefore, $z_1$ is a solution of a polynomial equation with degree at most $n+1$, and the proof is completed.

\vspace{5mm}
\begin{ack}
  This work is inspired by the lectures of Walter Neumann and many discussions with Christian Zickert.
  Also the author received many suggestions from Roland van der Veen, Jun Murakami, Hyuk Kim and Seonhwa Kim
  while preparing this article.  
  The author expresses his gratitude to them. 
  
  The author is supported by Grant-in-Aid for JSPS Fellows 21.09221.
\end{ack}

\bibliography{005}
\bibliographystyle{abbrv}

{\sc 
\begin{flushleft}
  Department of Mathematics, Faculty of Science and Engineering, Waseda University, 3-4-1 Okubo, Shinjuku-ku, Tokyo 169-8555, Japan\\
E-mail: dol0425@gmail.com
\end{flushleft}}
\end{document}